\documentclass
{amsart}

\usepackage{amssymb,amscd}
\usepackage[all,line,arc,curve,color,frame,pdf]{xy}
\usepackage{tikz}
\usepackage{tikz-cd}
\usetikzlibrary{positioning, trees, snakes}
\usepackage[textsize=tiny]{todonotes}
\usepackage{hyperref}
\usepackage[shortlabels]{enumitem}
\usepackage{float} 
\usepackage[paper=a4paper, margin=3.5cm]{geometry}
\usepackage{cleveref}  
\usepackage[symbol]{footmisc}
\usepackage{comment}
\usepackage{scalerel}
\usepackage[normalem]{ulem}

\numberwithin{equation}{section}
\theoremstyle{plain}
\newtheorem{thm}{Theorem}[section]

\newtheorem{lemma}[thm]{Lemma}
\newtheorem{prop}[thm]{Proposition}

\theoremstyle{remark}
\newtheorem{rem}[thm]{Remark}
\newtheorem{ex}[thm]{Example}

\theoremstyle{definition}
\newtheorem{defi}[thm]{Definition}

\def\ot{\otimes}
\def\boxot{\boxtimes}
\def\fA{A}
\def\fV{\mathcal{V}}
\def\mm{\mathfrak{m}}
\def\SG{S}
\def\NN{\mathbb{N}}
\DeclareMathOperator{\brk}{\underline{rk}}
\DeclareMathOperator*{\bigbot}{\scalerel*{\boxtimes}{\sum}}

\usepackage[textsize=tiny]{todonotes}

\usepackage{enumitem}

\newcommand\ignore[1]{}

\newcommand\CC{{\mathbb{C}}}

\newcommand\QQ{{\mathbb{Q}}}
\newcommand\ZZ{{\mathbb{Z}}}

\def\operatorname#1{\mathop{\rm #1}\nolimits}

\def\Spec{\operatorname{Spec}}

\def\rk{\operatorname{rk}}
\def\supp{\operatorname{supp}}

\def\deg{\operatorname{deg}}

\def\End{\operatorname{End}}

\newcommand{\pb}{\ar@{}[dr]|{\text{\pigpenfont J}}}

\makeatletter
\newcommand{\xleftrightarrow}[2][]{\ext@arrow 3359\leftrightarrowfill@{#1}{#2}}
\makeatother
\newcommand{\xdasharrow}[2][->]{
\tikz[baseline=-\the\dimexpr\fontdimen22\textfont2\relax]{
\node[anchor=south,font=\scriptsize, inner ysep=1.5pt,outer xsep=2.2pt](x){#2};
\draw[shorten <=3.4pt,shorten >=3.4pt,dashed,#1](x.south west)--(x.south east);
}}

\begin{document}
\title[Bounds on complexity of matrix multiplication away from CW tensors]{Bounds on complexity of matrix multiplication away from Coppersmith--Winograd tensors}

\author[Homs]{Roser Homs}
\address{Technical University of Munich, Department of Mathematics, Parkring 13, 85748 Garching bei M\"unchen, Germany}
\email{roser.homs@tum.de}
\author[Jelisiejew]{Joachim Jelisiejew}
\address{Faculty of Mathematics, Informatics and Mechanics, University of Warsaw, Banacha 2, 02-097 Warsaw}
\thanks{JJ was partially supported by Polish National Science Center, project
2017/26/D/ST1/00755 and by the START
fellowship of the Foundation for Polish Science.}
\email{jjelisiejew@mimuw.edu.pl}

\author[Micha{\l}ek]{Mateusz Micha{\l}ek}
\address{
	University of Konstanz, Germany, Fachbereich Mathematik und Statistik, Fach D 197
	D-78457 Konstanz, Germany
}
\thanks{MM is funded by the Deutsche Forschungsgemeinschaft –- Projektnummer 467575307.}
\email{mateusz.michalek@uni-konstanz.de}
\author[Seynnaeve]{Tim Seynnaeve}
\address{University of Bern, Mathematical Institute, Sidlerstrasse 5, 3012 Bern, Switzerland}
\email{tim.seynnaeve@math.unibe.ch}

\begin{abstract}
    We present three families of minimal border rank tensors: they come from
    highest weight vectors, smoothable algebras, and monomial algebras. We
    analyse them using Strassen's laser method and obtain an upper bound
    $2.431$ on $\omega$. We also explain how in certain monomial cases using
    the laser method directly is less profitable than first degenerating. Our
    results form possible paths in the search for valuable tensors for the
    laser method away from Coppersmith-Winograd tensors.
\end{abstract}

\maketitle

\section{Introduction}

 Determining the complexity of matrix multiplication is a central problem in computer science. Its algebraic counterpart translates to estimating the rank or border rank of the matrix multiplication tensor $M_{\langle n,n,n \rangle}\in\CC^{n^2}\ot\CC^{n^2}\ot\CC^{n^2}$; see \cite{BurgisserBook, JM1, landsberg_2017} or \cite[Chapter 9.3]{jaBernd}.
The complexity of matrix multiplication is measured by the constant $\omega$, defined as the smallest number such that for any $\epsilon >0$ the multiplication of $n\times n$ matrices can be performed in 
$O(n^{\omega+\epsilon})$ arithmetic operations. Equivalently, $\omega$ is the smallest number such that for any $\epsilon>0$ the rank (or border rank) of $M_{\langle n,n,n \rangle}$ is $O(n^{\omega+\epsilon})$.

The best known upper bounds on $\omega$ are all obtained using the so-called \emph{laser method}, which is based on the work of Strassen \cite{Strassen87}.
The idea behind the laser method is to, instead of studying the matrix multiplication tensor directly, consider a different tensor which can be proven to have low border rank, and at the same time is ``close" to being a matrix multiplication tensor. 
Strassen obtained a bound $\omega < 2.48$ using the laser method. Shortly thereafter, Coppersmith and Winograd introduced a new tensor, and applied (in a highly nontrivial way) the laser method to it to obtain $\omega < 2.3755$ \cite{CW}. Since then, the improvements on the bound of $\omega$ were made by Stothers, Williams, and Le Gall \cite{Stothers, Virgi, LeGall}, arriving at the current state of the art $\omega<2.373$. These improvements were all obtained by applying the laser method to the Coppersmith-Winograd tensor and its powers. However, several recent results~\cite{Alman_Williams, Ambainis_Filmus_LeGall, Christandl_Vrana_Zuiddam2} proved the existence of barriers: as a very particular case, using the Coppersmith-Winograd tensors alone in the laser method, one cannot obtain a upper bound for $\omega$ close to $2$.
 
Below we present other tensors appearing naturally and apply the laser method to find the bound on $\omega$ they give. We explore two approaches of constructing such tensors. One approach is based on the highest weight vectors in $S^3(\mathfrak{sl}_n)$. As we show, a few of them are the Coppersmith-Winograd tensors! We use the other highest weight vectors to prove bounds on $\omega$. For one of those tensors we obtain a bound $\omega < 2.45$, which is not as good as the Coppersmith-Winograd bound, but better than Strassen's bound.

Another approach builds upon the works \cite{LaMM} and \cite{Blaser}. Under certain genericity assumptions, a tensor is of minimal border rank if and only if it is the multiplication tensor of a smoothable finite-dimensional algebra. The Coppersmith-Winograd tensor arises in this way: it is the multiplication tensor of an algebra with Hilbert function $(1,n,1)$, and it follows from a result of Cartwright et al.\ \cite{CEVV09} that such an algebra is always smoothable. We give an example of an algebra with Hilbert function $(1,n,2)$ (hence smoothable by \cite{CEVV09}) whose multiplication tensor is suitable for the laser method. The obtained bound $\omega < 2.431$ is better than the bound above, but still not as good as the bound obtained from the Coppersmith-Winograd tensor.
We hope that our constructions will inspire other mathematicians to look for ``valuable" tensors different from the Coppersmith-Winograd tensor.

Finally, we look at a very simple algebra $\CC[x]/(x^2)$. The associated tensor may be regarded as a very degenerate Coppersmith-Winograd tensor (in $\CC^2\otimes \CC^2\otimes\CC^2$). By taking the third Kronecker power of that tensor, corresponding to the third tensor product of algebras, we obtain a new tensor $T$, for which we apply the laser method obtaining $\omega < 2.56$. Here however our main observation is that $T$ degenerates to the Coppersmith-Winograd tensor, which is used to obtain 
all the upper bounds on $\omega$ since 1988.
This means that directly applying the 
laser method 
to $T$ is very far from optimal. In particular, it suggests that a new method, suited to analyze tensors like $T$, is needed, and if found will provide new bounds on $\omega$.

\section{Preliminaries}
We start with presenting preliminaries about tensors and fast matrix
multiplication. We work in the space $\fV=U\ot V \ot W$, where $U,V,W$ are
finite-dimensional vector spaces over $\CC$. We assume familiarity with the
notions of rank and border rank of tensors (denoted by $\rk(T)$ and $\brk(T)$,
respectively) and (border) Waring rank of symmetric tensors, see \cite[\S2.4,~\S2.6]{JM1} for an in-depth treatment.
 
Let $G:=GL(U) \times GL(V) \times GL(W)$. There is a natural action of $G$ on $\fV$, which extends to an action of the algebra $\fA:=\End(U) \times \End(V) \times \End(W) \supseteq G$.
\begin{defi}
	Let $T,T' \in \fV$.
	\begin{itemize}
		\item We say $T'$ is a \emph{restriction} of $T$, denoted $T' \leq T$, if $T' \in \fA\cdot T$. 
		\item We say $T'$ is a \emph{degeneration} of $T$, denoted $T' \trianglelefteq T$, if $T' \in \overline{G \cdot T}$. 
	\end{itemize}
\end{defi}
Informally, $T' \leq T$ means that $T'$ can be obtained from $T$ by
restricting to subspaces of $U,V,W$ and doing a change of basis. Degeneration
is an approximate version of restriction: $T' \trianglelefteq T$ means that
$T'$ is a limit of restrictions of $T$, see for example~\cite[\S2.2]{Blaser}.

\begin{defi}
	The \emph{Kronecker product} $T \boxot T'$ of two tensors $T \in U \ot V \ot W$ and $T' \in U' \ot V' \ot W'$ is by definition their tensor product $T \ot T'$, viewed as a $3$-way tensor in $(U \ot U') \ot (V \ot V') \ot (W \ot W')$. The bracketing is important: the rank one tensors in this space are all tensors of the form $u'' \ot v'' \ot w''$, with $u'' \in U \ot U'$, $v'' \in V \ot V'$, $w'' \in W \ot W'$. 
\end{defi}
A tensor $T\in \CC^a\otimes\CC^b\otimes\CC^c$ may be identified with a linear map $\CC^a\rightarrow \CC^b\otimes\CC^c$, after identification of vector spaces with their duals when required.
The image of that map, which may be represented as a linear space of $b\times c$ matrices, determines $T$ up to isomorphism. Hence, we will often represent $T$ as a space $L_T$ of matrices, as in (\ref{eq:matrix}). 
The tensor $T$ is called \emph{concise} if the corresponding linear maps $\CC^a\rightarrow \CC^b\otimes\CC^c$, $\CC^b\rightarrow \CC^a\otimes\CC^c$, $\CC^c\rightarrow \CC^a\otimes\CC^b$ are injective. Every tensor can be made concise by replacing $\CC^a$, $\CC^b$, $\CC^c$ with suitable subspaces. The border rank of a concise tensor is at least $\max \{a,b,c\}$; in case of equality we say that $T$ has \emph{minimal border rank}.

We denote the $a \times b \times c$ \emph{matrix multiplication tensor} by $M_{\langle a,b,c \rangle} \in \CC^{ab} \ot \CC^{bc} \ot \CC^{ca}$, see \cite[Section 2.5.2]{JM1} and \cite[Chapter 9.3]{jaBernd} for details. To prove bounds on $\omega$, instead of analyzing $M_{\langle a,b,c \rangle}$ directly, we will consider different tensors that are both of low (typically minimal) border rank, and degenerate to a direct sum of many large matrix multiplication tensors. This second requirement is made quantitative using the notion of degeneracy value.

\begin{defi}[{\cite[Definition 3.4.7.1]{landsberg_2017}}]
		Let 
		$T \in U \ot V \ot W$ 
		be a tensor. For all $N \in \mathbb{N}$, we define 
		\[
		V_{\omega,N}(T) = \sup \left\{\sum_{i=1}^{q}{(a_ib_ic_i)^{\omega/3}} \middle| T^{\boxot N} \trianglerighteq \bigoplus_{i=1}^{q}M_{\langle a_i,b_i,c_i \rangle}\right\},
		\]
		where we take the supremum over all possible ways of degenerating $T^{\boxot N}$ to a direct sum of matrix multiplication tensors. 
		The \emph{degeneracy value} (or simply \emph{value}) $V_{\omega}(T)$ of $T$ is defined as the supremum $\sup_N{V_{\omega,N}(T)^{\frac{1}{N}}}$.
\end{defi}
\begin{rem}
	Here are some easy but important properties of the degeneracy value. The first two are trivial; for the third one, see \cite[Section 3.4.7]{landsberg_2017}.
	\begin{itemize}
		\item If $T \trianglerighteq T'$, then $V_{\omega}(T) \geq V_{\omega}(T')$.
		\item Supermultiplicativity: $V_{\omega}(T \boxot T') \geq V_{\omega}(T)V_{\omega}(T')$.
		\item Superadditivity: $V_{\omega}(T \oplus T') \geq V_{\omega}(T)+V_{\omega}(T')$.
	\end{itemize}
\end{rem}
The following theorem is a restatement of the results by Strassen and Sch\"onhage's asymptotic sum inequality ({see \cite{Schoenhage81} or \cite[(15.11)]{BurgisserBook}}).

\begin{thm} \label{thm:asymsum}
	For any tensor $T$, we have $V_{\omega}(T) \leq \brk(T)$.
\end{thm}
\begin{proof}
For any $N$ consider any degeneration $T^{\boxot N} \trianglerighteq
\bigoplus_{i=1}^{q}M_{\langle a_i,b_i,c_i \rangle}$. It induces
the following inequalities on ranks:
\[\brk\left(\bigoplus_{i=1}^{q}M_{\langle a_i,b_i,c_i \rangle}\right)\leq \brk(T^{\boxot N})\leq \brk(T)^N.\]
By Sch\"onhage's asymptotic sum inequality we obtain:
\[\sum_{i=1}^{q}{(a_ib_ic_i)^{\omega/3}}\leq\brk(T)^N.\]
By considering the supremum over all such degenerations of $T^{\boxot N}$ we get $V_{\omega,N}(T)\leq \brk(T)^N$ for every $N$. Taking the $N$-th root and passing to the supremum over all $N$ we obtain the result.
\end{proof}
Therefore, any given tensor $T$ with an upper bound on its border rank and a lower bound on its value --- in terms of the constant $\omega$ --- yields an upper bound on $\omega$, as illustrated in \Cref{ex:strassen}.
In the rest of this section we briefly recall Strassen's \emph{laser method}, which can be used to estimate the value of a tensor if the tensor can be decomposed into smaller tensors for which the value is known. The version we present here was first introduced in \cite{thesisTim}.


\begin{defi}
	Let $T \in U \ot V \ot W$ be a tensor. 
	\begin{itemize}
		\item A \emph{blocking} $D$ of $T$ is given by decompositions 
		$U= \bigoplus_{i \in I}{U_{i}}$, $V= \bigoplus_{j \in J}{V_{j}}$, $W= \bigoplus_{k \in K}{W_{k}}$.
		These induce a decomposition 
		\[
		T=	\sum_{(i,j,k) \in I \times J \times K}{T_{(i,j,k)}}.
		\]
		\item The \emph{support} of a blocking $D$, denoted $\supp_DT$, consists of all triples $(i,j,k) \in I \times J \times K$ for which $T_{(i,j,k)} \neq 0$.
		\item We say $\supp_DT$ is \emph{tight}, if there are injections $\alpha: I \to \mathbb{Z}^r$,  $\beta: J \to \mathbb{Z}^r$,  $\gamma: K \to \mathbb{Z}^r$ s.t.\ $\alpha(i)+\beta(j)+\gamma(k)=0$ for all $(i,j,k) \in \supp_DT$. 
		\item The \emph{symmetrization} of $T$, denoted by $\tilde{T}$, is the Kronecker product $T \boxot T' \boxot T'' \in (U \ot V \ot W)^{\ot 3}$, where $T' \in V \ot W \ot U$ and $T'' \in W \ot U \ot V$ are obtained from $T$ by cyclically permuting the indices. 
	\end{itemize}	
\end{defi}
It is not hard to see that $\brk(\tilde{T})\leq \brk(T)^3$ and $V_\omega(\tilde{T}) \geq V_{\omega}(T)^3$.
As opposed to \cite[Theorem 4.1]{LeGall}, instead of assuming lower bounds on the values of the blocks, we will only assume known bounds on the values of their symmetrizations.

\begin{rem}
	We apply the laser method to tensors $T\in \CC^m\otimes\CC^m\otimes\CC^m$ of minimal border rank, where the blocking $D$ is given by partitioning the basis elements of each $\CC^m$ into three groups labeled by $\{0,1,2\}$. We assume that the tensor has nonzero entries only if the labels belong to the groups with indices summing up to $2$, i.e. if $(i,j,k) \in \supp_DT$ then $i+j+k=2$. For instance, under our assumptions $T_{(1,1,0)}$ is typically nonzero, while $T_{(1,0,0)}=0$. Such a blocking is automatically tight: we can take $r=1$ and define $\alpha(i)=i$, $\beta(j)=j$, $\gamma(k)=2-k$. 
\end{rem}

Let $D$ be a blocking of a tensor $T\in U \ot V \ot W$, with indexing sets $I,J,K$. We consider probability distributions $P: \supp_DT \to [0,1]$ on $\supp_DT$ (this gives each block a ``weight"). We write $P_I: I \to [0,1]$ for the marginal distribution on $I$, and similarly for $J$, $K$. The entropy $H(P_I) := -\sum_{i \in I}{P(i) \log P(i)}$ (and similarly for $J$, $K$) plays an important role in the laser method, see Lemma~\ref{lem:entropy}.
We make a technical (but easy to verify) assumption on our blocking which removes the need for the extra term $-\Gamma_S(P)$ appearing in \cite[Theorem 4.1]{LeGall}:

\begin{defi} \label{def:reconstructible}
	We say that a subset $\Phi \subseteq I \times J \times K$ is \emph{reconstructible}, if every probability distribution on $\Phi$ is uniquely determined by its 3 marginal distributions.
\end{defi}
\begin{ex}
	Let $I=J=K=\{0,1,2\}$, and take 
	\[
	\Phi=\{(2,0,0),(0,2,0),(0,0,2),(1,1,0),(1,0,1),(0,1,1)\}.
	\] 
	Then $\Phi$ is tight and reconstructible: if $P$ is a probability distribution on $\Phi$, then the equalities $P(2,0,0)=P_I(2)$, $P(0,1,1)=P_I(0)-P_J(2)-P_K(2), \ldots$  allow us to reconstruct $P$ from the marginal distributions $P_I,P_J,P_K$.
	However, if we consider 
	\[
	\Phi'=\{(2,1,0),(1,2,0),(2,0,1),(1,0,2),(0,2,1),(0,1,2)\},
	\] 
	then $\Phi'$ is again tight, but not reconstructible. For instance, the uniform distribution $P(x,y,z)=\frac{1}{6}$ and the distrubution $P(0,1,2)=P(2,0,1)=P(1,2,0)=\frac{1}{3}$ have the same marginals.
\end{ex}
\begin{thm}[Laser method] \label{thm:laserMethod}
	Let $T \in U \otimes V \otimes W$, and let $D$ be a blocking of $T$, indexed by $I \times J \times K$. Assume that $\supp_DT$ is tight and reconstructible. 
	Let $P$ be any probability distribution on $\supp_DT$.
	We have an inequality
	\begin{equation} \label{eq:laser}
	\log{V_{\omega}(\tilde{T})} \geq H(P_I)+H(P_J)+H(P_K)+\sum_{\supp_DT}{P(i,j,k)\log{V_\omega(\widetilde{T_{(i,j,k)}})}}.
	\end{equation}
\end{thm}
\begin{proof}
	See \Cref{sec:appendix}.
\end{proof}

\begin{ex}[Strassen's tensor] \label{ex:strassen}
	Consider the following tensor:
	\[
	T=T_{STR,n}:=\sum_{i=1}^{n}{u_0 \ot v_i \ot w_i} + \sum_{i=1}^{n}{u_i \ot v_0 \ot w_i} \in U \ot V \ot W,
	\]
	where $U$, $V$ and $W$ have respective bases $\{u_0,u_1,\ldots,u_n\}$, $\{v_0,v_1,\ldots,v_n\}$, and $\{w_1,\ldots,w_n\}$. We consider the block decomposition $D$ with $I=J=\{0,1\}$, $K=\{1\}$, $U_0=\langle u_0 \rangle$, $U_1=\langle u_1, \ldots u_n \rangle$, $V_0=\langle v_0 \rangle$, $V_1=\langle v_1, \ldots v_n \rangle$, $W_1=\langle w_1,\ldots, w_n \rangle$. Then the support $\supp_DT=\{(0,1,1),(1,0,1)\}$ is clearly tight and reconstructible.\\
	Since $T_{(0,1,1)}=\sum_{i}{u_0 \ot v_i \ot w_i}=M_{\langle 1,1,n\rangle}$, we find that $\widetilde{T_{(0,1,1)}}=M_{\langle n,n,n\rangle}$. Similarily we find that $\widetilde{T_{(1,0,1)}}=M_{\langle n,n,n\rangle}$. Let $P$ be the uniform distribution on $\{(0,1,1),(1,0,1)\}$, then $H(P_1)=H(P_2)=\log(2)$ and $H(P_3)=0$. \Cref{thm:laserMethod} now gives
	\[
	\log V_{\omega}(\tilde{T}) \geq 2\log(2) + \log(n^{\omega}),
	\]
	so that $V_{\omega}(\tilde{T}) \geq 4n^{\omega}$.
	The border rank of $T$ is equal to $n+1$:  it cannot be lower than $n+1$, as $T$ is a concise tensor; and on the other hand we can write $T$ as a limit
	\[
	\lim_{t \to 0}{\frac{1}{t}\Big(\sum_{j=1}^{n}{(u_0 + t u_i)\ot(v_0 + t v_i)\ot w_i} - u_0 \ot v_0 \ot (w_1 + \cdots + w_n)\Big)}.
	\]
	Hence the border rank of $\widetilde{T}$ is at most $(n+1)^3$, so \Cref{thm:asymsum} yields
	\[
	4n^{\omega} \leq (n+1)^3.
	\]
	For $n=5$, this gives Strassen's bound $\omega < 2.48$ \cite{Strassen87}.
\end{ex}

\section{Tensors from highest weight vectors}
In \cite{Plethysm}, the fourth author studied the highest weight vectors of the $\mathfrak{sl}_n$-representation $S^3(\mathfrak{gl}_n)$. In this section, we show that several of these highest weight vectors can be identified with the Coppersmith-Winograd tensor. We then study one of the other highest weight vectors, argue it is well-suited for the laser method, and obtain a bound $\omega < 2.45$ from it. 

Given elements $x,y,z$ in a vector space $V$, $xyz:=\frac{1}{6}(x\otimes y\otimes z + x\otimes z\otimes y + y\otimes x\otimes z + y\otimes z\otimes x + z\otimes x\otimes y + z\otimes y\otimes x)$ is a 3-way symmetric tensor in $V\otimes V\otimes V$.
Similarly, $xy:=\frac{1}{2}(x\otimes y + y\otimes x)$ is 
a 2-way symmetric tensor in $V\otimes V$.

We recall the small and big Coppersmith-Winograd tensors $T_{cw,m}$ and $T_{CW,m}$ from \cite{CW}. They are symmetric tensors, precisely elements of $\CC[x_0,\ldots,x_{m+1}]_3$, given by 
\begin{eqnarray*}
	T_{cw,m}=\sum_{i=1}^{m}{3}{x_0x_i^2} & \text{and} & T_{CW,m}={3}x_0^2x_{m+1}+{3}\sum_{i=1}^{m}{x_0x_i^2}
\end{eqnarray*}
with border Waring rank equal to $m+2$.

A basis of $\mathfrak{gl}_n$ is given by $\lbrace E_{i,j}\rbrace_{1\leq i,j\leq n}$, where $E_{i,j}$ is the matrix with a $1$ at position $(i, j)$ and $0$ elsewhere. One of the highest weight vectors in $S^3(\mathfrak{gl}_n)$ is the symmetrized matrix multiplication tensor $\sum_{i,j,k}{E_{i,j}E_{j,k}E_{k,i}}$ \cite{chiantini2017polynomials}. 
The other highest weight vectors are listed in \cite[Table 2]{Plethysm}.
An interesting observation is that many of these highest weight vectors 
are, up to a change of variables, equal to $T_{cw,m}$ or $T_{CW,m}$ for some value of $m$.
\begin{prop} \label{prop:HWVCW}
	The following equalities hold, up to a change of variables:
	\begin{eqnarray*}
		IE_{1,n}E_{2,n-1}-IE_{1,n-1}E_{2,n} &=& T_{cw,4}\\
		E_{1,n}E_{1,n-1}E_{2,n} - E_{1,n}E_{1,n}E_{2,n-1} &=& T_{CW,2}\\
		\sum_i{IE_{1,i}E_{i,n}} &=& T_{cw,2n-2}\\
		\sum_i{E_{1,n}E_{1,i}E_{i,n}} &=& T_{CW,2n-4}\\
		\sum_{i,j}{E_{1,n}E_{i,j}E_{j,i}} &=& T_{CW,n^2-2},
	\end{eqnarray*}
	i.e.~five of the highest weight vectors are Coppersmith-Winograd tensors.
\end{prop}
\begin{proof}
	We only prove the fourth equality; the other ones are similar and left to the reader.
	Note that
	\[
	\sum_{j=1}^n{E_{1,n}E_{1,j}E_{j,n}} = E_{1,n}^2(E_{1,1}+E_{n,n}) + \sum_{j=2}^{n-1}{E_{1,n}E_{1,j}E_{j,n}}.
	\]
	If we substitute $E_{1,n}=3x_0$, $E_{1,1}+E_{n,n}=\frac{1}{3}x_{2n-3}$, $E_{1,j}=x_{2j-3}+ix_{2j-2}$ and $E_{j,n}=x_{2j-3}-ix_{2j-2}$ for $2 \leq j \leq n-1$ (where $i$ is the imaginary unit, so that $E_{1,j}E_{j,n}=x_{2j-3}^2+x_{2j-2}^2$), we obtain the Coppersmith-Winograd tensor $T_{CW,2n-4}$.
\end{proof}

We now focus on one of the other highest weight vectors from \cite[Table 2]{Plethysm}, that is not a Coppersmith-Winograd tensor.
Let $n \geq 3$ and write
\[
 T_{HW,n} := \sum_{i=1}^n (E_{1,n}E_{2,i}E_{i,n} - E_{2,n}E_{1,i}E_{i,n}) \in S^3(\CC^{n^2}).
\]
This is not a concise tensor, but we can make it concise by changing the ambient space: first, rewrite $T_{HW,n}$ as
\[
E_{1,n}^2E_{2,1} - E_{2,n}^2E_{1,2} + E_{1,n}E_{2,n}(E_{2,2}-E_{1,1})+ \sum_{i=3}^{n-1} (E_{1,n}E_{2,i}E_{i,n} - E_{2,n}E_{1,i}E_{i,n}).
\]
For every $i \in \{3, \ldots, n-1\}$, we put $E_{i,n}=x_{i-3}$, $E_{2,i}=y_{i-3}$, and $E_{1,i}=z_{i-3}$.
Moreover, we put $E_{1,1}-E_{2,2}=b_0$,  $E_{1,n}=a_1$, $E_{2,n}=-a_2$, $E_{2,1}=b_1$, $E_{1,2}=-b_2$. Then, after suitably rescaling, our tensor becomes
\[
T_{HW,m}=6a_1a_2b_0+ 3a_1^2b_1 +3a_2^2b_2 + 6\sum_{j=1}^{m} (a_1x_jy_j + a_2x_jz_j) \in S^3(V) \subseteq V \ot V \ot V,
\]
where $m=n-3$ and $V$ is the $(3m+5)$-dimensional vector space with basis \[
\{a_1, a_2, x_1, \ldots, x_m,y_1, \ldots, y_m,z_1, \ldots, z_m, b_0, b_1, b_2\}.
\]
The space $L_{T_{HW,m}}$, 
made of all possible contractions of the tensor $T_{HW,m}$, consists of all matrices of the form 

\begin{equation}\label{eq:matrix}
\left(
\begin{array}{cc|ccc|ccc|ccc|ccc}
b_1 & b_0 & y_1 & \ldots  & y_m & x_{1} & \ldots & x_{m} & 0 & \ldots & 0 & a_2 & a_1 & 0\\
b_0 & b_2 & z_1 & \ldots  & z_m & 0  & \ldots & 0 & x_1 & \ldots & x_m & a_1 & 0 & a_2\\
\hline
y_1 &z_1 &  &   &  & a_{1} & &  & a_{2} &  &  &  & &\\
\vdots&\vdots &&&&& \ddots&&&\ddots&&&&\\
y_m&z_m &  &   &  &  & & a_{1} &  &  & a_{2} &  && \\
\hline
x_1&0 & a_{1}  &   &  &  & &  &  &  &  &  && \\
\vdots&\vdots&&\ddots&&&&&&&&&&\\
x_m &0 &  &   & a_{1} &  & &  &  &  &  &  && \\
\hline
0&x_1 & a_{2}  &   &  &  & &  &  &  &  &  && \\
\vdots&\vdots&&\ddots&&&&&&&&&&\\
0&x_m &  &   & a_{2} &  & &  &  &  &  &  &&\\
\hline
a_2&a_1&&&&&&&&&&&\\
a_1&0&&&&&&&&&&&\\
0&a_2&&&&&&&&&&&
\end{array}
\right).
\end{equation}
 
It is easy to see that $T_{HW,m}$ is concise, as the induced map $V^{*} \to V \ot V$ is injective. From this, it follows that $T_{HW,m}$ has border rank at least $\dim V =3m+5$. We are particularly interested in this highest weight vector because it has minimal border rank.
\begin{thm} \label{conj:minbrank}
	The border rank of $T_{HW,m}$ is equal to $3m+5$.
\end{thm}

\begin{proof}

The case $m=0$ follows from \Cref{prop:wrk}.
For all $m\geq 1$, consider the following collection of rank 1 symmetric tensors in $V\otimes V$:

\begin{itemize}

\item $A_{i,t}:=\left(t^{-1}a_1+tx_i+t^2y_i\right)^2=t^{-2}a_1^2+2a_1x_i+2ta_1y_i+t^2x_i^2+2t^3x_iy_i+t^4y_i^2$,

\item $B_{i,t}:=\left(t^{-1}a_2+tx_i+t^2z_i\right)^2=t^{-2}a_2^2+2a_2x_i+2ta_2z_i+t^2x_i^2+2t^3x_iz_i+t^4y_i^2$, 

\item $C_{i,t}:=\left(\sqrt{2}t^{-1}a_1/2+\sqrt{2}t^{-1}a_2/2+\sqrt{2}tx_i\right)^2$,

\item $D:=(a_1-a_2)^2=a_1^2+a_2^2-2a_1a_2$, 

\item $E_{1,t}:=\left(t^{-2}a_1+t^2 \sum_{j=1}^m x_j+t^3 \sum_{j=1}^m y_j+t^5 b_0-t^5 b_1\right)^2$,

\item $E_{2,t}:=\left(t^{-2}a_2+t^2 \sum_{j=1}^m x_j+t^3 \sum_{j=1}^m z_j+t^5 b_0-t^5 b_2\right)^2$,

\item $F_t:= \left(\sqrt{2}t^{-2}a_1/2+\sqrt{2}t^{-2}a_2/2+\sqrt{2}t^2\sum_{j=1}^m x_j+2\sqrt{2} t^5b_0\right)^2$,

\item $G_t:=\left(t^{-4}-m t^{-2}-1/2-l\right)a_1^2+\left(m t^{-2}-t^{-4}-1/2-l\right)a_2^2+\left(2l-1\right)a_1a_2$, 
\end{itemize}
with $l:=(mt^2-1)^2/(2t^8)$.

\medskip

We want to prove that $\lim_{t\rightarrow 0}\langle A_{i,t},B_{i,t},C_{i,t},D,E_{1,t},E_{2,t},F_t,G_t\rangle=L_{T_{HW,m}}$.

\medskip

Let $L_{T_{HW,m}}\vert_{x_i=1}$ denote the matrix in $L_{T_{HW,m}}$ obtained by specializing $x_i$ to $1$ and the rest of the variables to $0$.
We use similar notation for other variables. We can build
$3m+3$ generators of the $(3m+5)$-dimensional space of matrices $L_{T_{HW,m}}$ as follows:

\begin{itemize}
\item $\lim\limits_{t\to 0} A_{i,t}-t^2 E_{1,t}= 2a_1x_i=L_{T}\vert_{y_i=1}$,
\medskip
\item $\lim\limits_{t\to 0} B_{i,t}-t^2 E_{2,t}= 2a_2x_i=L_{T}\vert_{z_i=1}$,
\medskip
\item $\lim\limits_{t\to 0}t^{-1}\left(A_{i,t}+B_{i,t}-C_{i,t}-(1/2t^2)D\right)=2a_1y_i+2a_2z_i=L_{T}\vert_{x_i=1}$.
\item $\lim\limits_{t\to 0} t^2 A_{i,t}= a_1^2=L_{T}\vert_{b_1=1}$,
\medskip
\item $\lim\limits_{t\to 0} t^2 B_{i,t}= a_2^2=L_{T}\vert_{b_2=1}$,
\medskip
\item $\lim\limits_{t\to 0} t^2 A_{i,t}+t^2 B_{i,t}-D=2a_1a_2=L_{T}\vert_{b_0=1}$.
\end{itemize}




\medskip

The remaining two can be obtained as follows:

\begin{align*}
    \lim_{t\to 0} &\ t^{-3}\left(\sum_{j=1}^m
    A_{j,t}+\sum_{j=1}^m B_{j,t}-\sum_{j=1}^m
    C_{j,t}-E_{1,t}-E_{2,t}+\left(\frac{1}{2t^4}-\frac{m}{2t^2}\right)D+F_t\right)=\\
    & 2a_1b_0+2a_1b_1+2a_2b_0+2a_2b_2+2\sum_{j=1}^m x_jy_j+2\sum_{j=1}^m
    x_jz_j=L_{T}\vert_{a_1=a_2=1}\\
    \lim_{t\to 0} &\ t^{-3}\left(\sum_{j=1}^m A_{j,t}-\sum_{j=1}^m
    B_{j,t}-E_{1,t}+E_{2,t}+t^4F_t+lD+G_t\right)=\\
    &-2a_1b_0+2a_1b_1+2a_2b_0-2a_2b_2+2\sum_{j=1}^m x_jy_j-2\sum_{j=1}^m
    x_jz_j=L_{T}\vert_{a_1=-a_2=1}.
\end{align*}


\end{proof}

Even stronger, we conjecture that the border \emph{Waring} rank of $T_{HW,m}$ is equal to $3m+5$. For $m=0,1$ we have the following symmetric border decompositions.
\begin{prop}\label{prop:wrk}
	For $m = 0$ and $m=1$, the tensor $T_{HW,m}$ has border Waring rank $3m+5$.
\end{prop}
\begin{proof}
We provide explicit Waring rank approximations in both cases. For the case $m=0$, let 
\begin{eqnarray*}
	T_{0,t} = 3(a_1+tb_1)^3
	+6(a_2+tb_2)^3
	+(a_1-2a_2)^3\\
	-3(a_1-a_2+tb_0)^3
	-(a_1+a_2-3tb_0)^3.
\end{eqnarray*}
	Then taking the limit we obtain
	\[
	\lim_{t \to 0}{\frac{T_{0,t}}{t}} = 36a_1a_2b_0+9a_1^2b_1+18a_2^2b_2,
	\]
	which is equal to our tensor $T_{HW,1}$ up to rescaling the variables.
	For the case $m=1$, let
\begin{eqnarray*}
T_{1,t}=(-a_1-a_2+t^3b_0)^3+
\frac{1}{3}(-a_1+a_2)^3+
(a_1-t^2y_1+t^3b_1)^3+\\
\frac{1}{3}(a_1-a_2+3tx_1-3t^3b_0)^3+
\frac{1}{4}(2a_2-2tx_1-t^2z_1+t^3b_2)^3+\\
(-a_1-2tx_1+t^2y_1)^3+
\frac{1}{4}(-2a_2+t^2z_1)^3+
(a_1+a_2+tx_1)^3.\\
\end{eqnarray*}
Then we obtain
\[
\lim_{t \to 0}{\frac{T_{1,t}}{t^3}} = 12a_1a_2b_0+3a_1^2b_1+3a
_2^2b_2+12a_1x_1y_1+6a_2x_1z_1,
\]
which is equal to $T_{HW,1}$ up to rescaling the variables.
\end{proof}

\begin{rem}
    One approach to generalizing Proposition~\ref{prop:wrk} to all $m$ would be using
    border apolarity~\cite{buczBucz_border}. It would be enough to find a
    homogeneous ideal
    \[
        I \subset \CC[a_1, a_2, x_1, \ldots, x_m,y_1, \ldots,
        y_m,z_1, \ldots, z_m, b_0, b_1, b_2]
    \]
        such that $I \subset
    T_{HW,m}^{\perp}$ and prove that $I$ is a limit of saturated ideals of
    points. However, the ideal $I$ is not
    uniquely determined by $T_{HW,m}$. One possible choice of $I$ would be
    $I_0$ given by all quadrics annihilating $T_{HW,m}$ and additionally by
    the element $a_1^4a_2-a_1a_2^4$ and by
    \[
        (a_1^3-a_2^3)x_i\mbox{ for all } i =1, \ldots ,m.
    \]
    This ideal has the correct Hilbert function.
    While the ideal $I_0$ is indeed a limit of saturated ideals for small values of
    $m$, we do not know whether it is so for all $m$.
\end{rem}

We now apply the laser method to the tensor $T_m = T_{HW,m}$. The blocking $D$ is given by $V=V_0 \oplus V_1 \oplus V_2$, with
\[
V_0=\langle a_1,a_2 \rangle, V_1=\langle x_1, \ldots x_m, y_1, \ldots, y_m, z_1, \ldots, z_m \rangle, V_2=\langle b_0,b_1,b_2 \rangle
\]
The support 
of $T_m$ 
is tight: $\{(0,0,2),(0,2,0),(2,0,0),(1,1,0),(1,0,1),(0,1,1)\}$. The block
\begin{equation}\label{eq:doubleStrassen}
T_{m,(1,1,0)} = \sum_{j=1}^{m}{(x_j \ot y_j \ot a_1 + x_j \ot z_j \ot a_2 + y_j \ot x_j \ot a_1 + z_j \ot x_j \ot a_2)},
\end{equation}
%
%
%
can be identified with the Kronecker product 
\[
\left(\sum_{j=0}^{m-1}{e_j \ot e_j \ot 1}\right) \boxot (e_0 \ot e_1 \ot e_1 + e_0 \ot e_2 \ot e_2 + e_1 \ot e_0 \ot e_1 + e_2 \ot e_0 \ot e_2).
\]
The first factor is the matrix multiplication tensor $M_{\langle 1,m,1 \rangle}$, whose symmetrization (with respect to $\ZZ/{3\ZZ} \subset \SG_3$) has value $V_{\omega}(M_{\langle m,m,m \rangle}) = m^{\omega}$. The second factor is Strassen's tensor $T_{STR,2}$ from \Cref{ex:strassen}. We have $V_{\omega}(\widetilde{T_{STR,2}}) \geq 4\cdot2^{\omega}$, and 
\[
V_{\omega}(\widetilde{T_{m,(1,1,0)}}) \geq 4(2m)^{\omega}.
\]

The same holds for $T_{m,(1,0,1)}$ and $T_{m,(0,1,1)}$. 
To compute the value of the remaining blocks, i.e $T_{m,(0,0,2)}$, $T_{m,(2,0,0)}$ and $T_{m,(0,2,0)}$, we apply the laser method to the symmetrization of the tensor
\begin{equation*}
T_{m,(0,0,2)}=\resizebox{.15\hsize}{!}{$
\left(
\begin{array}{c|c}
b_1 & b_0 \\
\hline
b_0 & b_2 \\
\end{array}
\right),$}
\end{equation*}
and the blocking as above, where additionally the letters are in three separate groups. 
Equivalently $T_{m,(0,0,2)}=a_1 \ot a_1 \ot b_1 + a_1 \ot a_2 \ot b_0 + a_2 \ot a_1 \ot b_0 + a_2 \ot a_2 \ot b_2$, and the blocks are simply the summands in this expression. The value of (the symmetrization of) each block is equal to $1$. Hence by assigning the blocks $a_1 \ot a_1 \ot b_1$ and $a_2 \ot a_2 \ot b_2$ a probability $\frac{1}{3}$, and the remaining two blocks a probability $\frac{1}{6}$, the inequality (\ref{eq:laser}) becomes 
\[
\log{V_{\omega}(\widetilde{T_{m,(0,0,2)}})} \geq \log(2) + \log(2) + \log(3) + 0,
\]
so we obtain that this symmetrization $\widetilde{T_{m,(0,0,2)}}$ has value at least $12$.

With the above value estimates, \Cref{thm:laserMethod} yields
\begin{multline*}
3\log(3m+5)=3\log\brk(T_m) \geq \log{V_{\omega}(\widetilde{T_m})} \geq\\ 
{H(P_1)+H(P_2)+H(P_3)}
+\left(P(2,0,0)+P(0,2,0)+P(0,0,2)\right)\log(12)\\
 + \left(P(1,1,0)+P(1,0,1)+P(0,1,1)\right)\log\left(4(2m)^{\omega}\right) \text{.}
\end{multline*}
We obtain for every $m$ a bound on $\omega$, by optimizing over all probability distributions on $\supp_DT_m$. For $m=7$ we obtain 
$\omega < 2.45$. We point out that our analysis may be improved.

\section{Tensors from smoothable algebras}
In this section we present a second approach for finding new tensors suitable for the laser method, namely by considering multiplication tensors of smoothable algebras. We present a concrete example and obtain a bound $\omega <2.431$. 

An algebra $A$ is called \emph{smoothable} if $\Spec(A)$ is a smoothable scheme. Its multiplication map $V \ot V \to V$ (where $V$ is the underlying vector space of $A$) can be seen as a tensor $T_A \in V^* \ot V^* \ot V$.
\begin{prop}[{\cite[Corollary 3.6]{Blaser}}] \label{prop:smoothablebrank}
	If $A$ is a smoothable algebra, then the multiplication map of $A$ is a tensor of minimal border rank.
\end{prop}
\begin{thm}[{\cite[Propositions 4.12 and 4.13]{CEVV09}}]\label{thm:hilbsmooth}
	If $A$ is a local algebra with Hilbert function $(1,n,1)$ or $(1,n,2)$, then $A$ is smoothable.
\end{thm}
\begin{ex} \label{eg:CWalgebra}
	The Coppersmith-Winograd tensor arises as the multiplication tensor of a smoothable algebra. Let $V$ be an ($n+2$)-dimensional vector space with basis $\{a_0,\ldots,a_{n+1}\}$, and let $A=(V,*)$ be the algebra with unit $a_0$ defined by $a_i*a_i=a_{n+1}$ for all $i=1,\ldots, n$, $a_i*a_j=0$ for $1\leq i < j \leq n+1$, and $a_{n+1}*a_{n+1}=0$.
	\begin{rem}
	We would get an isomorphic algebra if we defined $a_0$ to be the unit, $a_i*a_j=a_{n+1}$ for any $i+j=n+1$ and $0$ otherwise. In particular, we would get isomorphic tensors.
	\end{rem}
	Equivalently we have $A=\CC[a_1, \ldots, a_{n+1}]/I$, where $I$ is the ideal generated by the relations above. 
	The multiplication tensor of $A$ is equal to 
	\[
	\alpha_0 \ot \alpha_0 \ot a_0 + \sum_{i=1}^{n+1}(\alpha_0 \ot \alpha_i \ot a_i+\alpha_i \ot \alpha_0 \ot a_i) + \sum_{i=1}^{n}(\alpha_i \ot \alpha_i \ot a_{n+1}),
	\]
	where $\lbrace\alpha_0,\dots,\alpha_{n+1}\rbrace$ is the dual basis of $\{a_0,\ldots,a_{n+1}\}$. This is the Coppersmith-Winograd tensor $T_{CW,n}$.
	Now, $A$ is a local graded algebra, with $\mm=(x_1,\ldots,x_{n+1})$ and grading given by $\deg(x_i)=1$ for all $i=1,\ldots,n$ and $\deg(x_{n+1})=2$, so the Hilbert function of $A$ is equal to $(1,n,1)$. Hence by \Cref{thm:hilbsmooth} $\Spec(A)$ is a smoothable scheme, and by \Cref{prop:smoothablebrank} $T_{CW,m}$ is of minimal border rank.
\end{ex}
\begin{ex} 
	Consider the $3m+3$-dimensional algebra $A=\CC[a_1,\ldots,a_{3m+2}]/I$, where the ideal $I$ is generated by 
\begin{itemize}
	\item $a_{3m+1}-a_ia_{m+i}$ for $1 \leq i \leq m$,
	\item $a_{3m+2}-a_ia_{2m+i}$ for $1 \leq i \leq m$,
	\item and all products $a_ia_j$, $1 \leq i \leq j \leq 3m+2$ not occuring in one of the above generators.
\end{itemize}
Its Hilbert function is given by $(1,3m,2)$.
As in the previous example, we conclude that $\Spec(A)$ is a smoothable scheme, and hence the multiplication tensor $T_A$ is a tensor of minimal border rank $3m+3$.

Explicitly, $T_A$ is equal to the following tensor in $A^* \ot A^* \ot A$ (with $A \cong \CC^{3m+3}$):
\begin{multline*}
T_A=\alpha_0 \ot \alpha_0 \ot a_0 + \sum_{i=1}^{3m}{(\alpha_0 \ot \alpha_i \ot a_i + \alpha_i \ot \alpha_0 \ot a_i)}\\
+ \sum_{i=1}^{m}(\alpha_i \ot \alpha_{m+i} \ot a_{3m+1} + \alpha_i \ot \alpha_{2m+i} \ot a_{3m+2} + \alpha_{m+i} \ot \alpha_i \ot a_{3m+1} + \alpha_{2m+i} \ot \alpha_i \ot a_{3m+2})\\
+ \alpha_0 \ot \alpha_{3m+1} \ot a_{3m+1} + \alpha_0 \ot \alpha_{3m+2} \ot a_{3m+2}+ \alpha_{3m+1} \ot \alpha_0 \ot a_{3m+1}+ \alpha_{3m+2} \ot \alpha_0 \ot a_{3m+2}.
\end{multline*}

The space $L_{T_A}$ consists of all matrices of the form 

\begin{equation*}
\resizebox{.9\hsize}{!}{$
\left(
\begin{array}{c|ccc|ccc|ccc|cc}
\lambda_0 & \lambda_1 & \ldots  & \lambda_m & \lambda_{m+1} & \ldots & \lambda_{2m} & \lambda_{2m+1} & \ldots & \lambda_{3m} & \lambda_{3m+1} & \lambda_{3m+2}\\
\hline
\lambda_1 &  &   &  & \lambda_{3m+1} & &  & \lambda_{3m+2} &  &  &  & \\
\vdots &&&&& \ddots&&&\ddots&&&\\
\lambda_m &  &   &  &  & & \lambda_{3m+1} &  &  & \lambda_{3m+2} &  & \\
\hline
\lambda_{m+1} & \lambda_{3m+1}  &   &  &  & &  &  &  &  &  & \\
\vdots&&\ddots&&&&&&&&&\\
\lambda_{2m} &  &   & \lambda_{3m+1} &  & &  &  &  &  &  & \\
\hline
\lambda_{2m+1} & \lambda_{3m+2}  &   &  &  & &  &  &  &  &  & \\
\vdots&&\ddots&&&&&&&&&\\
\lambda_{3m} &  &   & \lambda_{3m+2} &  & &  &  &  &  &  & \\
\hline
\lambda_{3m+1}&&&&&&&&&&&\\
\lambda_{3m+2}&&&&&&&&&&&
\end{array}
\right).$}
\end{equation*}
We point out that this looks exactly like the ``multiplication table'' of $A$. 
Also, note that if we take the ideal generated by the $2 \times 2$ minors of this matrix and substitute $\lambda_0=1$ and $\lambda_i=a_i$ for $i>0$, we recover $I$.

We now apply the laser method to $T_A$.
We take the blocking $D$ given by
\[
A^*_0=\langle \alpha_0 \rangle, A^*_1 = \langle \alpha_1, \ldots, \alpha_{3m} \rangle, A^*_2 = \langle \alpha_{3m+1}, \alpha_{3m+2} \rangle
\]
for the first 2 tensor factors, and
\[
A_0=\langle a_{3m+1}, a_{3m+2} \rangle, A_1 = \langle a_1, \ldots, a_{3m} \rangle, A_2 = \langle a_0 \rangle
\]
for the third tensor factor.
The support $\supp_DT_A$ is  again equal to
\[
\{(0,0,2),(0,2,0),(2,0,0),(1,1,0),(1,0,1),(0,1,1)\},
\]
which is tight.
The blocks $T_{A,(2,0,0)} = M_{\langle 2,1,1 \rangle}$, $T_{A,(0,2,0)} = M_{\langle 1,2,1 \rangle}$, $T_{A,(0,0,2)} = M_{\langle 1,1,1 \rangle}$, $T_{A,(0,1,1)} = M_{\langle 1,3m,1 \rangle}$ and $T_{A,(1,0,1)} = M_{\langle 1,1,3m \rangle}$ are all matrix multiplication tensors, and hence their values are known. 
The final block
\begin{multline*} 
T_{A,(1,1,0)} = \sum_{i=1}^{m}(\alpha_i \ot \alpha_{m+i} \ot a_{3m+1} + \alpha_i \ot \alpha_{2m+i} \ot a_{3m+2}\\ + \alpha_{m+i} \ot \alpha_i \ot a_{3m+1} + \alpha_{2m+i} \ot \alpha_i \ot a_{3m+2})
\end{multline*}
is precisely the same as the block (\ref{eq:doubleStrassen}) from previous section. Hence its value is at least $4(2m)^{\omega}$.
We now apply the laser method using this value estimate:
\begin{multline*}
3\log\brk(T_A) \geq \log{V_{\omega}(\widetilde{T_A})} \geq {H(P_1)+H(P_2)+H(P_3)} + P(1,1,0)\log\left(4(2m)^{\omega}\right)\\
+\left(P(2,0,0)+P(0,2,0)\right)\log(2^\omega) + \left(P(1,0,1)+P(0,1,1)\right)\log\left((3m)^{\omega}\right).
\end{multline*}
For fixed $m$ we can obtain a bound on $\omega$, by optimizing over all probability distributions on $\supp_DT_A$. 
The best bound is obtained by putting $m=4$: we obtain $\omega < 2.431$.
\end{ex}

\section{Tensor products of monomial algebras}
Monomial algebras are a special class of smoothable algebras. This last section is devoted to the study of a special family of tensors arising from such an algebra. In particular, we observe that in this case the laser method fails to account for degenerations to the largest possible matrix multiplication tensors.

To a polynomial $f\in \CC[x_1,\dots,x_n]$ we associate the apolar ideal $f^\perp\subset\CC[\partial_{x_1},\dots,\partial_{x_n}]$ of differentials that annihilate $f$ and the algebra $\CC[\partial_{x_1},\dots,\partial_{x_n}]/f^\perp$. 
The last family of tensors we would like to discuss are simply powers of the algebra $\CC[x]/(x^2)$ which is obtained as an apolar to $f=x$. Taking Kronecker powers of the tensor corresponds to the tensor product of algebras.
\begin{lemma}
If two algebras are apolar respectively to $f_1(x_1,\dots,x_n)$ and $f_2(y_1,\dots,y_m)$ then their tensor product is the algebra apolar to the product $f_1(x_1,\dots,x_n)f_2(y_1,\dots,y_m)$.
\end{lemma}
\begin{proof}
    We write $f := f_1(x_1,\dots,x_n)f_2(y_1,\dots,y_m)$ and let $S_1 := \CC[\partial_{x_1}, \ldots ,\partial_{x_n}]$, $S_2 :=
    \CC[\partial_{y_1}, \ldots ,\partial_{y_m}]$ and $S =
    S_1\otimes_{\CC} S_2 = \CC[\partial_{x_1}, \ldots
    ,\partial_{x_n},\partial_{y_1}, \ldots ,\partial_{y_m}]$.
    We have $S_1(f_2^{\perp}) + S_2(f_1^{\perp}) \subset f^{\perp}$ so that we
    get a surjection $\pi\colon (S_1/f_1^{\perp})\otimes_{\CC}
    (S_2/f_2^{\perp}) \to
    S/f^{\perp}$. But $S/f^{\perp}  \simeq S f$ as vector spaces and we
    see directly that $\dim(S f) = \dim(S_1f_1)\cdot \dim(S_2f_2)$,
    so $\pi$ is an isomorphism.
\end{proof}
Let $A_1=\CC[x]/(x^2)$. Then $A_2:=A\otimes A$ is apolar to the quadric $xy$. 
In particular, it coincides with the Coppersmith-Winograd tensor $T_{CW,2}$. The third tensor power is $A_3:=A_2\otimes A_1$, which is apolar to the cubic $xyz$.
In matrix representation we have 
\begin{equation*}
T_{CW,2}^{\ot 3}=
\resizebox{.5\hsize}{!}{$
	\left(
	\begin{array}{cc|cc||cc|cc}
	a_{0,0} & a_{0,1} & a_{1,0} & a_{1,1} & b_{0,0} & b_{0,1} & b_{1,0} & b_{1,1}\\
	a_{0,1} & 0 & a_{1,1} & 0 & b_{0,1} & 0 & b_{1,1} & 0\\
	\hline
	a_{1,0} & a_{1,1} & 0 & 0 & b_{1,0} & b_{1,1} & 0 & 0\\
	a_{1,1} & 0 & 0 & 0 & b_{1,1} & 0 & 0 & 0\\
	\hline 
	\hline
	b_{0,0} & b_{0,1} & b_{1,0} & b_{1,1} & 0 & 0 & 0 & 0\\
	b_{0,1} & 0 & b_{1,1} & 0 & 0 & 0 & 0 & 0\\
	\hline
	b_{1,0} & b_{1,1} & 0 & 0 & 0 & 0 & 0 & 0\\
	b_{1,1} & 0 & 0 & 0 & 0 & 0 & 0 & 0\\
	\end{array}
	\right).$}
\end{equation*}
 Switching rows and colums 3 and 4, we obtain:
\begin{equation*}
\resizebox{.4\hsize}{!}{$
\left(
\begin{array}{c|ccc|ccc|c}
a & b & c & d & e & f & g & h\\
\hline
b & 0 & e & f & 0 & 0 & h & 0\\
c & e & 0 & g & 0 & h & 0 & 0\\
d & f & g & 0 & h & 0 & 0 & 0\\
\hline
e & 0 & 0 & h & 0 & 0 & 0 & 0\\
f & 0 & h & 0 & 0 & 0 & 0 & 0\\
g & h & 0 & 0 & 0 & 0 & 0 & 0\\
\hline
h & 0 & 0 & 0 & 0 & 0 & 0 & 0\\
\end{array}
\right).$}
\end{equation*}
The blocking above is not reconstructible, but we can analyze it using \cite[Theorem 4.1]{LeGall}. 
The only nontrivial value is that of the block  
\begin{equation*}
\resizebox{.15\hsize}{!}{$
\left(
\begin{array}{c|cc}
0 & e & f \\
\hline
e & 0 & g \\
f & g & 0 \\
\end{array}
\right),$}
\end{equation*}
for which we rerun the laser method using the blocking as above. In the end, the obtained bound is $\omega<2.56$. 

However, what is more interesting is that $A_3\neq T_{CW,6}$, but $A_3$ degenerates to $T_{CW,6}$. In the matrix representation above we can multiply rows and columns $2,3,4$ and $8$ by $t$, and the letters $b,c,d$ and $h$ by $t^{-1}$ (this amounts to acting with diagonal matrices on the corresponding tensor factors). Letting $t \to 0$ we obtain 
\begin{equation*}
	\resizebox{.4\hsize}{!}{$
		\left(
		\begin{array}{c|ccc|ccc|c}
		a & b & c & d & e & f & g & h\\
		\hline
		b & 0 & 0 & 0 & 0 & 0 & h & 0\\
		c & 0 & 0 & 0 & 0 & h & 0 & 0\\
		d & 0 & 0 & 0 & h & 0 & 0 & 0\\
		\hline
		e & 0 & 0 & h & 0 & 0 & 0 & 0\\
		f & 0 & h & 0 & 0 & 0 & 0 & 0\\
		g & h & 0 & 0 & 0 & 0 & 0 & 0\\
		\hline
		h & 0 & 0 & 0 & 0 & 0 & 0 & 0\\
		\end{array}
		\right),$}
\end{equation*}
which is the matrix representation of $T_{CW,6}$. The Hilbert function of $A_3$ is $(1,3,3,1)$ whereas the Hilbert function of the local graded algebra corresponding to the Coppersmith-Winograd tensor $T_{CW,6}$ is $(1,6,1)$, see \Cref{eg:CWalgebra}.

 In particular, this proves that high powers of the $A_3$ tensor (and hence also $A_1$) degenerate to matrix multiplications that are much larger than suggested by the laser method. In particular, the laser method is very far from optimal when analyzing those tensors. We believe this calls for a new method. One of the reasons why laser method applied to the $A_3$ tensor is far from the optimal result is the fact that there are many blocks in $A_3$ that form bigger matrix multiplications (one would say that their phases are adjusted, although this is not required from the input of the algorithm).

\appendix
\section{Proof of \Cref{thm:laserMethod}} \label{sec:appendix}
We present here a proof of our version of the laser method. It is based on the proof of \cite[Theorem 15.41]{BurgisserBook} and very similar to the proof of \cite[Theorem 4.1]{LeGall}. The version (\Cref{thm:laserMethod2}) presented in this appendix is even a bit more general than \Cref{thm:laserMethod}: we consider symmetrization with respect to an arbitrary subgroup of $S_3$, instead of only $\ZZ/{3\ZZ}$. 
We need some preliminary definitions and results.
	\begin{defi}[{\cite[(15.29)]{BurgisserBook}}]
	Let $I,J,K$ be finite sets.
	If $\Psi \subseteq \Phi \subseteq I \times J \times K$, we call $\Psi$ a \emph{combinatorial degeneration} of $\Phi$, written $\Psi \trianglelefteq \Phi$, if there exist functions $\alpha: I \to \ZZ$, $\beta: J \to \ZZ$, $\gamma: K \to \ZZ$, such that $\alpha(i)+\beta(j)+\gamma(k)=0$ whenever $(i,j,k) \in \Psi$, and $\alpha(i)+\beta(j)+\gamma(k)>0$ whenever $(i,j,k) \in \Phi \setminus \Psi$.
\end{defi}
\begin{prop}[{\cite[(15.30)]{BurgisserBook}}] \label{prop: degen}
	Let $D$ be a blocking of $T$, with components $T_{(i,j,k)}$, $(i,j,k) \in \supp_DT \subseteq I \times J \times K$. Let $\Psi$ be a combinatorial degeneration of $\supp_DT$. Then we have
	\[
	\sum_{(i,j,k) \in \Psi}{T_{(i,j,k)}} \trianglelefteq T,
	\]
	where $\trianglelefteq$ denotes the usual tensor degeneration.
\end{prop}
\begin{defi}
Let $I,J,K$ be finite sets and $\Delta \subseteq I  \times J \times K$.
\begin{itemize}
	\item We say $\Delta$ is a \emph{diagonal}, if the three projections $\Delta \to I$, $\Delta \to J$, $\Delta \to K$ are injective.
	\item Recall that $\Delta$ is \emph{tight}, if there are injections $\alpha: I \to \mathbb{Z}^r$,  $\beta: J \to \mathbb{Z}^r$,  $\gamma: K \to \mathbb{Z}^r$ s.t.\ $\alpha(i)+\beta(j)+\gamma(k)=0$ for all $(i,j,k) \in \Delta$. If moreover $\alpha, \beta, \gamma$ can be chosen such that their images are contained in $\{-b,-b+1,\ldots, b-1,b\}^r$, we say $\Delta$ is \emph{$b$-tight}.
	\item We say $\Delta$ is \emph{balanced}, if the projection $p_I: \Delta \to I$ is surjective, with all fibers of equal cardinality ${|\Delta|}/{I}$, and similar for the other projections $p_J$, $p_K$.
\end{itemize}
\end{defi}
The following theorem explains the relevance of tight sets: they are sets which are not quite diagonal, but contain a large diagonal which is a combinatorial degeneration. In particular, if a tensor $T$ has a tight blocking, it degenerates to a large direct sum of a subset of its blocks. Assuming lower bounds on the values of the blocks, this gives a lower bound on the value of $T$.
The balancedness assumption below is not essential: there is also a version of the theorem without it, but the statement is more complicated.
\begin{thm}[{\cite[(15.39), attributed to Strassen]{BurgisserBook}}] \label{thm:tight}
	There exists a constant $C_b$, only depending on $b$, such that every $b$-tight balanced subset $\Phi \subseteq I \times J \times K$ contains a diagonal of size at least $C_b \cdot \min\{|I|,|J|,|K|\}$, which is a combinatorial degeneration. 
\end{thm}
The entropy enters in the proof of \Cref{thm:laserMethod2} through the following lemma, which is an easy consequence of Stirling's formula.
\begin{lemma}[{See \cite[(15.40)]{BurgisserBook}}] \label{lem:entropy}
	Fix a finite set $S$.
	There exists a sequence $\rho_N$, with $\lim_{N\to \infty}{\rho_N}=0$, such that for every rational probability distribution $P$ on $S$ that can be written as $P(i)=Q(i)/N$ for some $Q: S \to \NN$, it holds that
	\[
	\left| \frac{1}{N} \log{\binom{N}{Q}} - H(P) \right| \leq \rho_N.
	\]
	Here $\log$ is the logarithm in base $e$, and $\binom{N}{Q}$ stands for the appropriate multinomial coefficient. Explicitly: if $S$ is the set $\{1,2,\ldots, n\}$, then we have $\binom{N}{Q} = \binom{N}{Q(1), \ldots ,Q(n)}$.
\end{lemma}
The symmetric group $\SG_3$ acts on $U \ot V \ot W$ by permuting the factors. For all $\sigma \in \SG_3$, the tensor $\sigma T \in  \sigma U \ot \sigma V \ot \sigma W$ has a blocking $\sigma D$, with components $(\sigma T)_{(\sigma(i), \sigma(j), \sigma(k))} = \sigma(T_{(i,j,k)})$.
We fix a subgroup $G \subseteq \SG_3$. 
For any tensor $T$, we denote its \emph{symmetrization} $\bigbot_{\sigma \in G}{\sigma T}$ by $\tilde{T}$.
Note that $V_{\omega}(\tilde{T}) \geq (V_{\omega}(T))^{|G|}$, by supermultiplicativity.

\begin{thm} \label{thm:laserMethod2}
	Let $T \in U \otimes V \otimes W$, and let $D$ be a blocking of $T$, indexed by $I \times J \times K$. Assume that $\supp_DT$ is tight and reconstructible. 
Let $P$ be any probability distribution on $\supp_DT$, let $G \subseteq \SG_3$ be a subgroup and let $\tilde{T}$ be the symmetrization of $T$ with respect to $G$.
Then the following inequality holds:
	\begin{equation} \label{eq:laser2}
	\log{V_{\omega}(\tilde{T})} \geq \min_{L \in \{I,J,K\}}{\sum_{\sigma \in G}{H(P_{\sigma L})}}+\sum_{\supp_DT}{P(i,j,k)\log{V_\omega(\widetilde{T_{(i,j,k)}})}}. 
	\end{equation}
\end{thm}
\Cref{thm:laserMethod} is the special case of the above theorem with $G=\ZZ/{3\ZZ}$.
\begin{proof}
	Assume that $D$ is $b$-tight.
	We assume that $P$ is a rational probablity distribution, i.e.\ $P(i,j,k)\in \QQ$ for all $(i,j,k)$. Since a general probability distribution can be approximated by rational ones, this suffices (see also the proof of \cite[(15,41)]{BurgisserBook}). 
	There is a map $Q: \supp_DT \to \NN$ and an $N \in \NN$ such that $P(i,j,k)=Q(i,j,k)/N$ for all $(i,j,k)$.
	Let $I_Q \subseteq I^N$ consist of all sequences in which the element $i$ appears exactly $Q_I(i):=N\cdot P_I(i)=\sum_{j,k}{Q(i,j,k)}$ times for all $i$. Note that $|I_Q|=\binom{N}{Q_I}$.
	We define $J_Q \subseteq J^N$ and $K_Q \subseteq K^N$ analogously. 
	
	
	Write $\tilde{T} := \bigbot_{\sigma \in G}(\sigma T)$, and consider the tensor 
	\[
	{\tilde{T}^{\boxot N}} = \bigbot_{\sigma \in G}(\sigma T)^{\boxot N} \in \Big(\bigbot_{\sigma \in G}(\sigma U)^{\boxot N}\Big) \ot \Big(\bigbot_{\sigma \in G}(\sigma V)^{\boxot N}\Big) \ot \Big(\bigbot_{\sigma \in G}(\sigma W)^{\boxot N}\Big).
	\]
	Now $\tilde{T}$ has a blocking $\tilde{D} := \bigbot_{\sigma \in G}(\sigma D)^{\boxot N}$, with support 
	\[
	\supp_{\tilde{D}}{{\tilde{T}^{\boxot N}}} = \prod_{\sigma \in G}{(\supp_{\sigma D}{\sigma T})^N}
	\subseteq (\prod_{\sigma \in G}{(\sigma I)^N}) \times (\prod_{\sigma \in G}{(\sigma J)^N}) \times (\prod_{\sigma \in G}{(\sigma K)^N}),
	\]
	which is again $b$-tight. 
	We define 
	\[
	\Phi := \left((\prod_{\sigma \in G}{\sigma I_Q}) \times (\prod_{\sigma \in G}{\sigma J_Q}) \times (\prod_{\sigma \in G}{\sigma K_Q})\right) \cap \supp_{\tilde{D}}{{\tilde{T}^{\boxot N}}}.
	\]
	It trivially holds that $\Phi \trianglelefteq \supp_{\tilde{D}}{{\tilde{T}^{\boxot N}}}$.\\
	Let $(x,y,z) \in (\prod_{\sigma \in G}{(\sigma I)^N}) \times (\prod_{\sigma \in G}{(\sigma J)^N}) \times (\prod_{\sigma \in G}{(\sigma K)^N})$, and write 
	\[
	(x,y,z) = \Big((i_{\sigma, \ell})_{\sigma \in G, 1 \leq \ell \leq N},(j_{\sigma, \ell})_{\sigma \in G, 1 \leq \ell \leq N},(k_{\sigma, \ell})_{\sigma \in G, 1 \leq \ell \leq N} \Big).
	\]
	From our reconstructibility assumption, it follows that $(x,y,z) \in \Phi$ if and only if for every $\sigma \in G$ and $(\sigma(i),\sigma(j),\sigma(k)) \in \supp_{\sigma D}{\sigma T}$, there are exactly $Q(i,j,k)$ indices $\ell$ for which $(i_{\sigma,\ell}, j_{\sigma, \ell}, k_{\sigma, \ell}) = (\sigma(i),\sigma(j),\sigma(k))$.
	We find
	\begin{multline*}
	{\tilde{T}^{\boxot N}}_{(x,y,z)} = \bigbot_{\sigma, \ell}{(\sigma T)_{(i_{\sigma, \ell},j_{\sigma, \ell},k_{\sigma, \ell})}}
	= \bigbot_{\scriptsize \begin{matrix} (i,j,k) \in \supp_DT \\ \sigma \in G \end{matrix}}{\sigma (T_{(i,j,k)})}^{Q(i,j,k)}\\
	=\bigbot_{(i,j,k) \in \supp_DT}{(\widetilde{T_{(i,j,k)}})^{Q(i,j,k)}},
	\end{multline*}
	hence by supermultiplicativity:
	\[
	V_{\omega}({\tilde{T}^{\boxot N}}_{(x,y,z)}) \geq \prod_{(i,j,k) \in \supp_DT}{V_{\omega}(\widetilde{T_{(i,j,k)}})^{Q(i,j,k)}}.
	\]
	
	We now apply Theorem \ref{thm:tight} to the balanced $b$-tight subset $\Phi \subseteq (\prod_{\sigma }{\sigma I_Q}) \times (\prod_{\sigma }{\sigma J_Q}) \times (\prod_{\sigma }{\sigma K_Q})$, and we find a diagonal $\Delta \trianglelefteq \Phi$, such that $|\Delta| \geq C_b \cdot \min_{L \in \{I,J,K\}}{\prod_\sigma{|\sigma L_Q|}}$.
	Since $\Phi \trianglelefteq \supp_{\tilde{D}}{{\tilde{T}^{\boxot N}}}$, we find $\Delta \trianglelefteq \supp_{\tilde{D}}{{\tilde{T}^{\boxot N}}}$. So by applying \Cref{prop: degen}, we find
	\[
	\bigoplus_{(x,y,z) \in \Delta}{{\tilde{T}^{\boxot N}}_{(x,y,z)}} \trianglelefteq {\tilde{T}^{\boxot N}}.
	\]
	So we find that
	\[
	V_{\omega}(\tilde{T}) \geq \Big( \sum_{(x,y,z)\in \Delta}{V_{\omega}({\tilde{T}^{\boxot N}}_{(x,y,z)})} \Big)^{\frac{1}{N}}\\
	\geq \Big( |\Delta|{ \prod_{(i,j,k) \in \supp_DT}{V_{\omega}(\widetilde{T_{(i,j,k)}})^{Q(i,j,k)}}} \Big)^{\frac{1}{N}} .
	\]
	By taking logarithms, we obtain 
	\[
	\log{V_{\omega}(\tilde{T})} \geq \frac{1}{N} \min_{L \in \{I,J,K\}} \log{\Big(C_b\cdot \prod_{\sigma}{\binom{N}{Q_{\sigma L}}}\Big)} + \sum_{\supp_DT}{P(i,j,k)\log{\Big(V_{\omega}(\widetilde{T_{(i,j,k)}})\Big)}}.
	\]
	Now the theorem follows by taking $N \to \infty$  and applying \Cref{lem:entropy}.
\end{proof}

\end{document}